\documentclass[11pt]{article}
\usepackage{mathrsfs}
\usepackage{amsmath}
\usepackage{color}
\usepackage{bbm}
\usepackage{amsmath,amsthm,amssymb,amscd}
\usepackage{latexsym}
\usepackage{hyperref}
\usepackage[numbers,sort&compress]{natbib}
\usepackage{hypernat}
\usepackage{indentfirst}

\newtheorem{theorem}{Theorem}[section]

\newtheorem{lemma}[theorem]{Lemma}

\newtheorem{definition}[theorem]{Definition}

\newtheorem{remark}[theorem]{Remark}

\newtheorem{assumption}[theorem]{ Assumption}
\numberwithin{equation}{section}

\begin{document}
	
	\title{{\bf  Stability of strong global and exponential  attractors for semilinear beam
			equations with fractional damping and
			memory}
		\thanks{The work was supported partly by the NSF of China
		(12171094), the Shanghai Key
		Laboratory for Contemporary Applied Mathematics (08DZ2271900), and Fuyang Normal University (2025KYQD0157,2025AHGXZK40563), Outstanding Youth Research Project of Anhui Universities (2023AH030075).}}
		\author{Yu-Ying Duan $^{a, b}$,  Ti-Jun Xiao $^{b}$\thanks{Corresponding author. E-mail: tjxiao@fudan.edu.cn (T.J. Xiao), 21110180054@alu.fudan.edu.cn (Y.Y. Duan).}\\
{\small $^a$ School of Mathematics and Statistics, Fuyang Normal University,}\\
	{\small Fuyang 236037, PR China}\\
	{\small $^b$ Shanghai Key Laboratory for Contemporary Applied Mathematics,}\\
	{\small School of Mathematical Sciences, Fudan University, Shanghai 200433, PR China}}

	\date{}
	\maketitle
	
\begin{abstract}
This paper investigates the stability of strong global and exponential attractors for a semilinear beam equation with memory and fractional damping, where $\alpha\in[0,2]$ denotes the fractional damping exponent and $\beta\in[0,1]$ the memory parameter. After showing the existence of a strong global attractor, we prove its upper semicontinuity in the parameter pair $(\alpha,\beta)$. We then
construct a family of strong exponential attractors and establish its continuity in $(\alpha,\beta)$.  Here, ``strong" means that the compactness, attraction, and parameter-robustness properties are established in a topology stronger than that of the phase space. Compared to existing $\beta = 0$ results, our findings hold in a stronger topology.
The analysis draws upon our recent higher-order regularity results for global attractors.
			\end{abstract}

	\vspace{0.4cm}
	
	
	\vspace{0.4cm}

	\section{Introduction}
  We are concerned with a family of semilinear beam equations with fractional damping and memory
\begin{equation}\label{1.1}
 	\left\{
 	\begin{array}{l}
 		u_{tt}+\Delta^{2}u+(-\Delta)^{\alpha}u_{t}-\beta\displaystyle\int_{0}^{\infty}g(s)\Delta^{2}  u(t-s)ds+f(u)=h(x), ~(x,t)\in \Omega\times R^{+},\\
   u=\Delta u=0, ~ (x,t)\in \Gamma \times R^{+},\\
   u(x,\tau)=u_{0}(x,\tau),~ u_{t}(x,0)=u_{1}(x),~ x\in \Omega, \tau\le 0.
 	\end{array}
 	\right.
 \end{equation}
 Here, $\Omega \subset \mathbb{R}^n$ is a bounded  domain with smooth boundary $\Gamma$. The terms $f(u)$ and $h(x)$ denote the nonlinear source and the external force, respectively. The fractional damping is given by $(-\Delta)^{\alpha} u_t$ with $\alpha \in [0,2]$, and the memory kernel $g$ is nonnegative and nonincreasing, with $u(t-s)$ representing the past history, where $\beta \in [0,1]$ is the memory damping coefficient. The memory term accounts for hereditary effects by incorporating the past deformation history of the material, while fractional damping provides a natural description of frequency-dependent structural dissipation. The interaction of these two mechanisms is therefore relevant to both the physical modeling of the beam and the analysis of its long-time dynamics.

The purpose of this paper is to prove the stability of strong attractor and exponential attractors with respect to perturbed parameter pair $(\alpha,\beta)$, therefore, we give the relation definitions. Let $X$ and $Y$ denote two complete metric spaces, $S(t)$ be a semigroup acting on $X$. For any bounded subset $B \subset X$, there exists a threshold $ T(B) > 0$ such that $ S(t)B \subset Y $ for all $t > T(B)$. That is, there exists an absorbing set for $(S(t), X)$ in $Y$.
\begin{definition}\rm \cite{Babin-1992} \label{def1.1}
\rm A set $A\subset X\cap Y$ is said to be a $(X, Y)$-global attractor of semigroup $S(t)$, if \rm{(i)} $A$ is bounded in $X$ and compact in $Y$; \rm{(ii)}  $A$ is invariant, i.e., $S(t)A=A, $ for any $t\ge 0$;  \rm{(iii)} it attracts every bounded subset $B$ of $X$ in the topology of $Y$, i.e., $ \lim_{t\to+\infty}\mbox{dist}_{Y}\{S(t)B, A\}=0.$
\end{definition}
\begin{definition}\rm \cite{Ding-2021}\label{def1.2}
    A set $A_{\exp}\subset X\cap Y$ is said to be a $(X, Y)$-exponential attractor of the semigroup $S(t)$, if \rm{(i)}  $A_{\exp}$ is bounded in $X$, while in $Y$ it is compact and has finite fractional dimension, i.e., $\dim_{f}(A_{\exp},Y)<+\infty$; \rm{(ii)}  $A_{\exp}$ is forward invariant, i.e., $S(t)A_{\exp}\subset A_{\exp} $ for any $t\ge 0$; \rm{(iii)}  $A_{\exp}$ attracts exponentially  every bounded subset $B$ of $X$ in the topology of $Y$, i.e.,  $ \mbox{dist}_{Y}\{S(t)B, A_{\exp}\}\le C\left(\|B\|_{X}\right)e^{-\gamma t}$ for some $\gamma>0$.
\end{definition}
It is evident that the aforementioned definitions align with the conventional ones when $X = Y$. When the topology of $Y$ is stronger than that of $X$, denoted as $Y \hookrightarrow X$, we characterize  $A$ as a strong  $(X, Y)$-global attractor, and $A_{\exp}$ as a strong  $(X, Y)$-exponential attractor. To investigate the stability of strong  attractors with respect to perturbations in the strong topology, we provide the following definition.
\begin{definition}\rm
    The family of strong $(X, Y)$-global or exponential attractors $\{A_{\eta}\}_{\eta\in I}$ is called
    \begin{enumerate}
        \item[\rm(1)] upper semicontinuous at $\eta_{0}\in I$ in $Y$-topology, if $\lim_{\eta\to \eta_{0}}\mbox{dist}_{Y}(A_{\eta}, A_{\eta_{0}})=0;$
        \item[\rm(2)] continuous at $\eta_{0}\in I$ in $Y$-topology, if $\lim_{\eta\to \eta_{0}}\mbox{dist}_{Y}^{symm}(A_{\eta}, A_{\eta_{0}})=0$.
   \end{enumerate}
 \par\noindent
 Here, the notations
        \begin{align*}
            \mbox{dist}_{Y}(A, B):=\sup_{x\in A}\inf_{y\in B}d_{Y}(x,y), ~\mbox{dist}_{Y}^{symm}(A,B):=\max\{\mbox{dist}_{Y}(A, B), \mbox{dist}_{Y}(B, A)\}
        \end{align*}
      denote the Hausdorff semi-distance and the Hausdorff distance between subsets $A$ and $B$ of $Y$, respectively.

\end{definition}

The long-time dynamics of evolution equations with memory has attracted considerable attention, owing to the nonlocal character of the hereditary term and the resulting lack of compactness in the associated history space. A substantial body of work has been devoted to the existence, regularity, and finite-dimensionality of global and exponential attractors for viscoelastic and hyperbolic equations with memory; see, for instance, \cite{Cavalcanti-2016,Guo, Kloeden-2011} and the references therein. Beyond the existence of attractors for a fixed memory kernel, an important issue is the robustness of the asymptotic dynamics under perturbations of the memory mechanism. In this direction, Conti, Pata, and Squassina \cite{Conti-2005} considered a rescaled family of kernels
$k_{\varepsilon}(s) = \frac{1}{\varepsilon} k\left(\frac{s}{\varepsilon}\right)$
which converges, in the distributional sense, to the Dirac mass at the origin as $\varepsilon\to0$, and established the convergence of the corresponding family of exponential attractors toward that of the limiting memoryless problem. Related singular-limit results for equations with fading memory have subsequently been obtained in different settings, including second-order evolution equations, weakly damped hyperbolic models, and wave equations with nonlinear memory terms \cite{Gat,qin,Zhang-2019}.

The scenario  $\beta = 0$ corresponds to fractional damped beam equations
\begin{align*}
	u_{tt}+\Delta^{2}u+(-\Delta)^{\alpha}u_{t}+f(u)=h(x), (x,t)\in \Omega\times R^{+}.
\end{align*}
For fractionally damped systems, considerable progress has been made in the theory of strong attractors and their parameter dependence. For the Kirchhoff wave model
\begin{align*}
u_{tt}-\phi(\|\nabla u\|^2)\Delta u
+\sigma(\|\nabla u\|^2)(-\Delta)^\theta u_t
+f(u)=g(x),
\end{align*}
Li and Yang \cite{Li-2020} proved that, for each $\theta\in[1/2,1)$, the associated semigroup possesses an optimal global attractor and an optimal exponential attractor, whose compactness, attraction, and finite-dimensionality hold in the stronger space $V_{1+\theta}\times V_{\theta}$ than the phase space $V_1\times L^2$ (see Section 2 for the definition of the space $V_s$). They also established the upper semicontinuity of the global attractor family with respect to $\theta$ in the corresponding strong topology. Subsequently, the works \cite{Ding-2022} and  \cite{LYD}  showed that these attractors are actually a strong $(V_1\times L^2,V_2\times V_{2\theta})$-global attractor and a strong $(V_1\times L^2,V_2\times V_{2\theta})$-exponential attractor.

For beam equations, Liu, Yang, and Guo \cite{Liu} considered the extensible beam model
\begin{align*}
	u_{tt}-\beta_0 M(\|\nabla u\|^2)\Delta u+\Delta^2u
	+(-\Delta)^{\alpha}u_t+f(u)=g(x),
\end{align*}
for $\alpha\in (0,1)$. They proved that the strong $(V_2\times L^2, V_4\times V_2)$-global attractor is upper semicontinuous with respect to $\alpha$ in $V_4\times V_2$, and with respect to $\beta_0$ in $V_{4} \times V_{2+\alpha_1}$ for  $\alpha_1 \in (0, \alpha/2)$. They also constructed strong exponential attractors that are continuous with respect to $\alpha$ in $V_4\times V_2$ and H\"older continuous with respect to $\beta_0$ in $V_4\times V_{2+\alpha_1}$. More recently, Ding, Jin, and Yang \cite{Ding} investigated
\begin{align*}
	u_{tt}+\Delta^2u+(-\Delta)^\alpha u_t
	+\Delta\phi(\Delta u)=g(x)
\end{align*}
for $\alpha\in(0,2]$, and obtained strong $
(V_{1+\alpha}\times L^{2}(\Omega),
V_{1+\alpha-\delta}\times V_{\alpha-\delta})
$-global and exponential attractors for any $0<\delta\ll 1$, together with the upper semicontinuity of the corresponding strong global attractor family in the associated strong topology. For further studies on strong attractors for beam equations, we refer to \cite{Li-2021,Yan}. Furthermore, the impact of different perturbation terms in the governing equations plays a crucial role in determining such continuity. For instance, Freitas et al. \cite{Freitas-2025} proved that, in the shear beam model, the limit $\kappa \to \infty$ of the shear elasticity modulus yields the Euler-Bernoulli beam model, and they established the upper semicontinuity of the global attractor with respect to $\kappa$. Similarly, Gomes Tavares et al. \cite{Gomes-2024} demonstrated the upper semicontinuity with respect to $\varepsilon$ for extensible beams subject to nonlinear energy damping of the form $(\mathcal{E}(U)+\varepsilon I)^{q}u_{t}$. For additional results on the continuity properties of perturbed attractors in dissipative dynamical systems, we refer the reader to \cite{Azevedo-2023, Araruna-2018, Freitas-2018}.

For the beam equation \eqref{1.1} considered in the present paper, under the optimal subcritical growth condition
$1 \le p < \min\{\frac{n+4\alpha}{(n-4)^{+}},\frac{n+4}{(n-4)^{+}}\}$ and
with $\alpha \in [0,2]$, the higher-order regularity of the global attractor was obtained in~\cite{Duan}. More precisely, with the natural energy space as the phase space, the attractor $\mathcal{A}_{\alpha,\beta}$ is bounded in
$V_{4} \times V_{4-\delta_{0}} \times \mathcal{M}_{4},$
for any $0 < \delta_{0} \ll 1$ (with $\mathcal M_s$ defined in Section 2).
Based on this regularity result and inspired by the work of \cite{Yang-2018,Liu,Ding}
, we further investigate the stability
of strong global and exponential attractors with respect to the
parameter pair \((\alpha,\beta)\). The
main results of this paper (some of them have been stated in \cite{Duan25}) are summarized as follows.
\begin{enumerate}
\item[\rm(i)]We establish the existence of a {\bf strong}
$(\mathcal H,\mathcal Z_{\varepsilon})$-global attractor (for any $0<\varepsilon\ll 1$), and prove its
upper semicontinuity with respect to the parameter pair $(\alpha,\beta)$
in a {\bf strong} topology. Here $\mathcal H$ is the natural energy space defined in Section 2 serving as the phase space, while $Z_{\varepsilon}$ is the regularized space defined in \eqref{z}.
 \item[\rm(ii)]
 We construct a family of {\bf strong}
 \((\mathcal H,\mathcal Z_{\varepsilon})\)-exponential attractors and
 establish its continuity with respect to the parameter pair
 \((\alpha,\beta)\) in a {\bf strong} topology, by applying the abstract framework of
 \cite{Yang-2018}.
\item[\rm(iii)]
For $\beta\equiv0$ (that is, the memory is absent) and $\alpha\in[0,2)$, corresponding to the phase space $V_2\times L^2$, we obtain a regularized space $V_4\times V_{4-\varepsilon}$ for the attractors of beam equations, with compactness and attraction, as well as upper semicontinuity of the global attractor and continuity of the exponential attractors with respect to
$\alpha$, established in its topology. For the same phase space, the regularized space provided here has a stronger topology than the known spaces for beam equations (cf. \cite{Liu,Ding}).

\end{enumerate}
{
The remainder of this paper is organized as follows. In Section 2, we present some preliminary results. In Section 3, we prove the upper semicontinuity of the strong global attractor with respect to the perturbation parameter pair $(\alpha, \beta)$. In Section 4, we establish the existence of a strong exponential attractor and its continuity with respect to the perturbation parameter pair $(\alpha,\beta)$.

Throughout this paper, $C(r)$ denotes a positive constant depending on $r$ while $C$ is a generic positive constant, and $\varepsilon, \varepsilon_{0}$ stand for small positive constants, which may vary from line to line.

\section{Preliminaries}
For convenience, we introduce the notation
\begin{align*}
V_{0}=L^{2}(\Omega),\quad V_{1}=H_{0}^{1}(\Omega),\quad V_{2}=H^{2}(\Omega)\cap H_{0}^{1}(\Omega).
\end{align*}
Define the operator $A:V_{2}\to V_{-2}$ (where $V_{-2}$ denotes the dual space of $V_{2}$) by
\begin{align*}
(Au, v)=(\Delta u, \Delta v),\quad \forall\, u,v\in V_{2}.
\end{align*}
The operator $A$ is self-adjoint in $V_{0}$ and positive definite on $V_{2}$.
For $\alpha\in[0,2]$, let $A^{\alpha}$ be the fractional power associated with $A$. We set $V_{s}:=D(A^{s/2})$ and endow it with the inner product and norm
\begin{align*}
	(u,v)_{V_{s}}=(A^{s/2}u,\, A^{s/2}v),\quad
	\|u\|_{V_{s}}=\|A^{s/2}u\|,
\end{align*}
where, $\|\cdot\|:=\|\cdot\|_{V_{0}}$.
Moreover, for $s\ge0$, we have the continuous embedding $V_{s}\hookrightarrow L^{\frac{2n}{n-2s}}$.

First, we present the following basic assumptions.
\begin{assumption}\label{ass1}\rm

 \textbf{ \rm{(i)} } $g:\lbrack 0,+\infty)\to (0,+\infty)$ is a decreasing and locally absolutely continuous function  which satisfies
 \begin{align}\label{g1}
   1- \beta G(0)=k>0,~~G(s):=\int_{s}^{\infty}g(\tau)d\tau\le \delta g(s).
 \end{align}
\textbf{ \rm{(ii)} } $f\in C^{1}(R)$ with $f(0)=0$, satisfying
        \begin{align}\label{f1}
            \liminf_{|s|\to \infty}\frac{f(s)}{s}>-k\lambda_{1},
        \end{align}
        where, $\lambda_{1}>0$ is the principal eigenvalue of operator $\Delta^{2} $ in $V_{2}$, and
        \begin{align}\label{f2}
            |f^{\prime}(s)|\le C(1+|s|^{p-1}),~ s\in R,
        \end{align}
    with
       \begin{equation}\label{assp}
       	1\le p<\left\{
       	\begin{array}{ll}
       		+\infty, ~~~~~~~~~~~~~~~~~~&n=1,2,3,4, \\
       		p^{*}:=\min\left\{\frac{n+4\alpha}{n-4}, \frac{n+4}{n-4}\right\},~ &n\ge 5.
       	\end{array}
       	\right.
       \end{equation}
\textbf{ \rm{(iii)} } $h(\cdot)\in L^{2}(\Omega)$.

\end{assumption}

As in \cite{Dafermos-1970}, we denote
\begin{align*}
    \eta(x,t,s)=u(x,t)-u(x,t-s).
\end{align*}
  We can transform \eqref{1.1} as
  \begin{equation}\label{1.2}
 	\left\{
 	\begin{array}{l}
 		u_{tt}+kA^{2}u+A^{\alpha}u_{t}+\beta\displaystyle\int_{0}^{\infty}g(s)A^{2}\eta(x,t,s)ds+f(u)=h(x),~(x,t)\in \Omega\times R^{+},\\
   \eta_{t}=-\eta_{s}+u_{t},~(x,t,s)\in \Omega\times R^{+}\times R^{+},\\
   u=\Delta u=\eta=\Delta \eta=0,~ (x,t)\in \Gamma \times R^{+}, s\in R^{+},\\
   u(x,\tau)=u_{0}(x,\tau),~ u_{t}(x,0)=u_{1}(x), ~x\in \Omega,~ \tau\le 0,\\
   \eta(x,0,s)=\eta_{0}(x,s), ~(x,s)\in\Omega\times R^{+}.
 	\end{array}
 	\right.
 \end{equation}
Let us define the history space with respect to memory kernel $g$ by
\begin{align*}
    \mathcal{M}_{s}=L_{g}^{2}(R^{+}, V_{s})=\left\{\eta:  R^{+}\rightarrow V_{s};~ \|\eta\|_{\mathcal{M}_{s}}<\infty\right\},
\end{align*}
 with inner and norm given by
 \begin{align*}
     \langle \eta,\zeta\rangle_{\mathcal{M}_{s}}=\int_{0}^{\infty}g(\tau)(\eta,\zeta)_{V_{s}}d\tau, ~~\|\eta\|_{\mathcal{M}_{s}}^{2}=\int_{0}^{\infty}g(\tau)\|\eta\|_{V_{s}}^{2}d\tau.
 \end{align*}
 Moreover, our phase space is
\begin{align*}
    \mathcal{H}=V_{2}\times V_{0}\times \beta\mathcal{M}_{2},
\end{align*}
which is equipped with the norm
\begin{align*}
    \|(u,u_{t},\eta)\|_{\mathcal{H}}^{2}=k\|\Delta u\|^{2}+\|u_{t}\|^{2}+\beta\|\eta\|_{\mathcal{M}_{2}}^{2}.
\end{align*}
Clearly, $\beta\mathcal{M}_{2}=\mathcal{M}_{2}$ for $\beta\in (0,1],$ while $\beta\mathcal{M}_{2}=\{0\}$ for $\beta=0.$  Additionally, we define regular (or weak regular) space by
 \begin{align*}
 	\mathcal{H}_{s}=V_{s+2}\times V_{s}\times \beta\mathcal{M}_{s+2},
 \end{align*}
 endowed with the norm
 \begin{align*}
 	\|(u,u_{t},\eta)\|_{\mathcal{H}_{s}}^{2}=k\|u\|_{V_{s+2}}^{2}+\|u_{t}\|_{V_{s}}^{2}+\beta\|\eta\|_{\mathcal{M}_{s+2}}^{2}.
 \end{align*}
Sometimes, we identify $V_{s+2}\times V_{s}$ with $V_{s+2}\times V_{s}\times\{0\}$.

As shown in \cite{Duan}, the existence and uniqueness of weak solution to problem \eqref{1.2} has been obtained by Faedo-Galerkin method.
	\begin{lemma}\textup{\cite{Duan}}\label{thm1}
	Suppose that assumption \ref{ass1} holds. Then, for every \(U_{0}\in\mathcal H\), problem \eqref{1.2} admits a
unique solution $
U(t)=S^{\alpha,\beta}(t)U_{0}.$
Furthermore, for each \(t>0\), the solution semigroup
\(S^{\alpha,\beta}(t)\) is \(1/2\)-H\"older continuous on \(\mathcal H\).
\end{lemma}
Meanwhile, the regularity of the attractor is also established.
\begin{lemma}\textup{\cite{Duan}}\label{them3.8}
	Suppose that Assumption \ref{ass1}  holds.
	Then, for any $\beta \in (0,1]$, the semigroup $\{S^{\alpha,\beta}(t)\}_{t \ge 0}$ admits a bounded absorbing set $B_{0}$ and a compact attractor $\mathcal{A}_{\alpha,\beta}$ on $  \mathcal{H}$, which is contained in and uniformly bounded in $\mathcal{Z} = V_{4} \times V_{4-\delta_{0}} \times \mathcal{M}_{4}$, where $\delta_{0} = 0$ for $\alpha = 2$ and $0 < \delta_{0} \ll 1$ for $\alpha \in [0,2)$. Moreover, when $\beta = 0$, the attractor  satisfies $\mathcal{A}_{\alpha,0} \subset V_{4} \times V_{4-\delta_{0}}$.
\end{lemma}

\section{Upper semicontinuity of the strong global attractor}

To establish the upper semicontinuity of strong attractors in this section, we require the following abstract criterion.

\begin{lemma}\textup{\cite{Rob-2001}}\label{lemma6.1}
	Assume that a dynamical system $(X,S^{\lambda}(t))$ possesses a compact global attractor $\mathcal{A}^{\lambda}$ for every $\lambda\in \Lambda$ (index set). Suppose further that the following conditions hold:
	\begin{enumerate}
		\item [{\rm(1)}] there exists a bounded subset $K\subset X$ such that $\bigcup_{\lambda\in \Lambda}\mathcal{A}^{\lambda}\subset K$;
		\item[{\rm(2)}] there exists a  $t_{0}>0$ such that
		\begin{align*}
			\lim_{\lambda\to \lambda_{0}}\sup_{x\in K}\mbox{d}_{X}(S^{\lambda}(t)x, S^{\lambda_{0}}(t)x)=0, ~\forall t\ge t_{0}.
		\end{align*}
	\end{enumerate}
	Then, the family of attractors $\mathcal{A}^{\lambda}$ is upper semicontinuous at $\lambda_{0}$, i.e.,
	\begin{align*}
		\lim_{\lambda\to \lambda_{0}}\operatorname{dist}_{X}\left(\mathcal{A}^{\lambda},\mathcal{A}^{\lambda{0}}\right)=0.
	\end{align*}
\end{lemma}

\begin{theorem}
  Let the assumption of  Lemma \ref{them3.8} hold. Then, for each fixed $$(\alpha,\beta) \in \Xi := [0,2] \times [0,1],$$ the global attractor $\mathcal{A}_{\alpha,\beta}$ of the semigroup $S^{\alpha,\beta}(t)$ is a strong $(\mathcal{H}, \mathcal{Z}_{\varepsilon})$-global attractor, where
  	\begin{equation}\label{z}
  		\mathcal{Z}_{\varepsilon} =\left\{
  		\begin{array}{ll}
  			V_{4-\epsilon} \times V_{4-\epsilon} \times \mathcal{M}_{4-\epsilon}, ~~&\mbox{if}~~0 < \beta \le 1, \\
  		V_{4} \times V_{4-\epsilon},\quad\quad \quad\quad\quad  &\mbox{if}~~\beta = 0~\mbox{and}~0\le \alpha<2,\\
  			V_{4-\epsilon}\times V_{4-\epsilon}, \quad\quad\quad\quad  &\mbox{if}~~\beta = 0~\mbox{and}~\alpha=2
  		\end{array}
  		\right.
  		\end{equation}
  with every sufficiently small $0 < \epsilon \ll 1$,  which may vary from line to line.
\end{theorem}
\begin{proof}
By Lemma \ref{them3.8}, the bounded absorbing set $B_0$ is uniformly bounded in $\mathcal{Z}$, and there exists $t_{B_0}>0$ such that  $S(t)B_0\subset B_0$ for all $t>t_{B_0}$. Set
\begin{align}\label{B}
	B_{\alpha,\beta}=\bigcup_{t\ge t_{B_{0}}+1}S^{\alpha,\beta}(t)B_{0}\subset B_{0},~~\mathcal{B}_{1}=\bigcup_{(\alpha,\beta)\in\Xi}B_{\alpha,\beta},~ ~\mathcal{B}=\left[\mathcal{B}_{1}\right]_{\mathcal{H}}.
\end{align}
Here, the sign $[\cdot]_{\mathcal{H}}$ denotes the closure in the space $\mathcal{H}$. Obviously, $	B_{\alpha,\beta},~\mathcal{B}_{1}$ and $\mathcal{B}$ are absorbing sets of the dynamical system $(\mathcal{H},S^{\alpha,\beta}(t))$ and  bounded in  $\mathcal{Z}$.

	Let
	\begin{align*}
		S^{\alpha,\beta}(t)\xi_{i}=(u^{i}(t), u_{t}^{i}(t),\eta^{i}(t)), ~\forall \xi_{i}\in \mathcal{B}, i=1,2.
	\end{align*}
Note  $z=u^{1}(t)-u^{2}(t),\zeta=\eta^{1}(x,t,s)-\eta^{2}(x,t,s)$ satisfy the following equation
 \begin{equation}\label{1.2--}
	\left\{
	\begin{array}{l}
	z_{tt}+kA^{2}z+A^{\alpha}z_{t}+\beta\displaystyle\int_{0}^{\infty}g(s)A^{2}\zeta(x,t,s)ds+f(u^{1})-f(u^{2})=0,~(x,t)\in \Omega\times R^{+},\\
		\zeta_{t}=-\zeta_{s}+z_{t},~(x,t,s)\in \Omega\times R^{+}\times R^{+},\\
	(z,z_{t},\zeta)|_{t=0}=\xi_{1}-\xi_{2}.
	\end{array}
	\right.
\end{equation}
	For  the case $0 < \beta \le 1$, by the interpolation inequality we have, for any  $0 < \epsilon_{i} \ll 1,~i=1,2$ and $0<\delta_{0}<\epsilon_{3}\ll 1$,
	\begin{align}\label{df}
		&\|z\|_{V_{4-\epsilon_{1}}}\le \|z\|_{V_{4}}^{\frac{2-\epsilon_{1}}{2}}\|z\|_{V_{2}}^{\frac{\epsilon_{1}}{2}}\le C(\mathcal{B})\|z\|_{V_{2}}^{\frac{\epsilon_{1}}{2}},\nonumber\\
		&\|\zeta\|_{\mathcal{M}_{4-\epsilon_{2}}}\le \|\zeta\|_{\mathcal{M}_{4}}^{\frac{2-\epsilon_{2}}{2}}\|\zeta\|_{\mathcal{M}_{2}}^{\frac{\epsilon_{2}}{2}}\le C(\mathcal{B})\|\zeta\|_{\mathcal{M}_{2}}^{\frac{\epsilon_{2}}{2}},\nonumber\\
		&	\|z_{t}\|_{V_{4-\epsilon_{3}}}\le \|z_{t}\|_{V_{4-\delta_{0}}}^{\frac{4-\epsilon_{3}}{4-\delta_{0}}}\|z_{t}\|^{\frac{\epsilon_{3}-\delta_{0}}{4-\delta_{0}}}\le C(\mathcal{B})\|z_{t}\|^{\frac{\epsilon_{3}-\delta_{0}}{4}}.
	\end{align}
	Without loss of generality, fix an arbitrary \(0<\epsilon\ll1\).
	Since \(\delta_0>0\) can be chosen arbitrarily small, we take
$0<\delta_0<\epsilon .$
	Applying the interpolation inequalities \eqref{df} with
$\epsilon_1=\epsilon_2=\epsilon_3=\epsilon,$
	we obtain
	\begin{align}\label{mu}
		\|S^{\alpha,\beta}(t)\xi_{1}-S^{\alpha,\beta}(t)\xi_{2}\|_{\mathcal{Z}_{\epsilon}}&\le C(\mathcal{B})	\|S^{\alpha,\beta}(t)\xi_{1}-S^{\alpha,\beta}(t)\xi_{2}\|_{\mathcal{H}}^{\frac{\epsilon_{*}}{4}}\nonumber\\
		&\le C(t)\|\xi_{1}-\xi_{2}\|_{\mathcal{H}}^{\frac{\epsilon_{*}}{8}}.
	\end{align}
	For $\epsilon_{*}=\epsilon-\delta_{0}$, $C(t)$ is a positive continous function depending on $t$.  That is, the mapping $S^{\alpha,\beta}(1):\mathcal{B}\subset \mathcal{H}\to \mathcal{Z}_{\epsilon}$ is $\frac{\epsilon_{*}}{8}-$H\"older continuous for each  $(\alpha,\beta)\in \Xi$; thus, $\mathcal{A}_{\alpha,\beta}=S^{\alpha,\beta}(1)\mathcal{A}_{\alpha,\beta}$ is a compact set in $\mathcal{Z}_{\epsilon}$. For every bounded subset $D$ of $\mathcal{H}$, there exists $t_{D}>0$ such that $	S^{\alpha,\beta}(\tau)D\subset \mathcal{B}$ for all $\tau\ge t_{D}$, by \eqref{mu}. For any $t>t_{D}+1$, we have
	\begin{align*}
		\mbox{dist}_{\mathcal{Z}_{\epsilon}}\left(S^{\alpha,\beta}(t)D, \mathcal{A}_{\alpha,\beta}\right)&=\mbox{dist}_{\mathcal{Z}_{\epsilon}}\left(S^{\alpha,\beta}(1)S^{\alpha,\beta}(t-t_{D}-1)S^{\alpha,\beta}(t_{D})D, S^{\alpha,\beta}(1)\mathcal{A}_{\alpha,\beta}\right)\\
		&\le C\left[\mbox{dist}_{\mathcal{H}}\left(S^{\alpha,\beta}(t-t_{D}-1)\mathcal{B}, \mathcal{A}_{\alpha,\beta}\right)\right]^{\frac{\epsilon_{*}}{8}}\to 0, \mbox{~as~}t\to \infty.
	\end{align*}
 Clearly, when $\beta = 0$ and  $\alpha =2$, we have
 	\begin{equation*}
 		\operatorname{dist}_{V_{4-\varepsilon} \times V_{4-\varepsilon}}
 		\left( S^{2,0}(t) D,\, \mathcal{A}_{2,0} \right) \to 0
 		\quad \text{as } t \to \infty.
 	\end{equation*}
 Furthermore, when \(\beta=0\) and \(\alpha\in[0,2)\), by using
 \eqref{df}, we first obtain	
\begin{align*}
	\|f(u^{1})-f(u^{2})\|\le C\left(1+\|u^{1}\|_{L^{\frac{(p-1)n}{4-\epsilon}}}^{(p-1)}+\|u^{2}\|_{L^{\frac{(p-1)n}{4-\epsilon}}}^{(p-1)}\right)\|z\|_{4-\epsilon}\le C(\mathcal{B})\|z\|_{V_{2}}^{\frac{\epsilon}{2}},
\end{align*}
and for the equation \eqref{1.2--}, there exists a sufficiently small $\theta_{*}>0$ and $0<\epsilon<2(2-\alpha)$ such that
	\begin{align*}
	\|z_{tt}\|_{V_{\alpha-2}}\le& C\left(\|A^{2}z\|_{V_{\alpha-2}}+\|A^{\alpha}z_{t}\|_{V_{\alpha-2}}+\|f(u^{1})-f(u^{2})\|_{V_{\alpha-2}}\right)\nonumber\\
	\le& C\left(\|z\|_{V_{2}}^{\frac{2-\alpha}{2}}\|z\|_{V_{4}}^{\frac{\alpha}{2}}+\|z_{t}\|_{V_{4-\epsilon}}+\|f(u^{1})-f(u^{2})\|\right)\nonumber\\
\le& C \left(\|z\|_{V_{2}}^{\theta_{*}}+\|z_{t}\|^{\theta_{*}}\right).
\end{align*}
Furthermore, since \(\|z_{tt}\|_{V_{2}}\) is uniformly bounded, the interpolation
inequality  gives
\begin{align}\label{3.5}
\|z\|_{V_{4}}\le& C\left(\|z_{tt}\|+\|z_{t}\|_{V_{2\alpha}}+\|f(u^{1})-f(u^{2})\|\right)\nonumber\\
\le&C\left(\|z_{tt}\|_{V_{2}}^{\frac{2-\alpha}{4-\alpha}}	\|z_{tt}\|_{V_{\alpha-2}}^{\frac{2}{4-\alpha}}+\|z_{t}\|_{V_{4-\epsilon}}+\|z\|_{V_{2}}^{\frac{\epsilon}{2}}\right)\nonumber\\
&\leq C(\mathcal B)\left(\|z\|_{V_{2}}^{\rho}+\|z_t\|^{\rho}\right),
\end{align}
for some sufficiently small \(\rho>0\). Combining the above estimates with the H\"older continuity of the
semigroup in the natural energy space, we conclude that $
	S^{\alpha,0}(1):\mathcal{B}\subset
	V_{2}\times L^{2}(\Omega)
	\longrightarrow
	V_{4}\times V_{4-\varepsilon}$
is H\"older continuous. Hence, we also get
\begin{equation*}
	\operatorname{dist}_{V_{4} \times V_{4-\varepsilon}}
	\left( S^{\alpha,0}(t) D,\, \mathcal{A}_{\alpha,0} \right) \to 0
	\quad \text{as } t \to \infty.
	\end{equation*}
To summarize, $\mathcal{A}_{\alpha,\beta}$ is a strong $(\mathcal{H}, \mathcal{Z}_{\epsilon})$ -global attractor for semigroup $S^{\alpha,\beta}(t)$.
\end{proof}

Based on the existence of the strong attractor established above, we  prove the upper semicontinuity of the strong  attractor with respect to $(\alpha_{0},\beta_{0} )\in\Xi$.
\begin{theorem}\label{them5}
	Let Assumption \ref{ass1}  hold, $(\alpha_{0},\beta_{0}) \in \Xi$, and let $$1 \leq p < p_{\alpha_0}: = \min\left\{\frac{n+4\alpha_0}{(n-4)^{+}}, \frac{n+4}{(n-4)^{+}}\right\}.$$ Then, the family of strong global attractors $\mathcal{A}_{\alpha,\beta}$ of the dynamical system $(\mathcal{H}, S^{\alpha,\beta}(t))$ is upper semicontinuous at the point $(\alpha_{0},\beta_{0})$:
	\begin{align}\label{3.3-}
		\lim_{(\alpha,\beta)\to(\alpha_{0},\beta_{0})} {\rm dist}_{V_{4-\epsilon} \times V_{4-\epsilon} \times \mathcal{M}_{4-\epsilon}}\{\mathcal{A}_{\alpha,\beta}, \mathcal{A}_{\alpha_{0},\beta_{0}}\}=0.
	\end{align}
In particular, for each fixed $\beta \in [0,1]$, the family of strong global attractors $\mathcal{A}_{\alpha,\beta}$ is upper semicontinuous at the point $\alpha_{0}\in\lbrack 0,2)$ in a stronger topology:
 \begin{equation}\label{3.4}
		\lim_{\alpha \to \alpha_{0}} \operatorname{dist}_{V_{4-\epsilon} \times V_{4} \times \mathcal{M}_{4}}
		\{\mathcal{A}_{\alpha,\beta},\,\mathcal{A}_{\alpha_{0},\beta}\} = 0.
	\end{equation}

\end{theorem}

\begin{proof}

	\textbf{(i)} In view of Lemma \ref{them3.8}, the  set $\mathcal{B}_{\alpha,\beta}$ is uniformly bounded in $\mathcal{Z}$ and the upper bound is independent of $\alpha$ and $\beta$.  Since  $(\alpha,\beta)\to (\alpha_{0},\beta_{0})$, we set $\Pi:= [a,b]\times[c,d],$ with
		\begin{align*}
		a:=\left\{
			\begin{array}{ll}
				\alpha_{0}-\delta_{1}, ~&\alpha_{0}\in(0,2\rbrack,\\
				0, ~~~~~ &\alpha_{0}=0,
			\end{array}
			\right.\quad
				b:=\left\{
			\begin{array}{ll}
				\alpha_{0}+\delta_{1}, ~&\alpha_{0}\in\lbrack0,2),\\
			2	, ~~~~~&\alpha_{0}=2,
			\end{array}
			\right.
		\end{align*}
		and
			\begin{equation*}
			c:=\left\{
			\begin{array}{ll}
				\beta_{0}-\delta_{1}, ~&\beta_{0}\in(0,1\rbrack,\\
			0, ~~~~~ &\beta_{0}=0,
			\end{array}
			\right.\quad
			d:=\left\{
			\begin{array}{ll}
				\beta_{0}+\delta_{1}, ~&\beta_{0}\in\lbrack0,1),\\
			1	, ~~~~~~ &\beta_{0}=1,
			\end{array}
			\right.
		\end{equation*}
where $\delta_{1} > 0$ is chosen sufficiently small such that $$0<\delta_{1}<\min\{\frac{\alpha_{0}}{2}, \quad 2-\alpha_{0}-\frac{\delta_{0}+\epsilon}{2}\}$$ with $$0<\epsilon<\min\{2-\alpha_{0},1-\beta_{0}\}.$$ Let
		\begin{align}\label{key1}
			B:= \bigcup_{(\alpha,\beta)\in \Pi}\mathcal{B}_{\alpha,\beta}, ~~\mathcal{B}_{R}=[B]_{\mathcal{H}}.
		\end{align}
	Then, $B_{R}$ is a compact set in $\mathcal{H}$ and  bounded in $\mathcal{Z}$.
	
		\textbf{(ii)} Let $$S^{\alpha,\beta}(t)\xi=(u^{\alpha,\beta}, u_{t}^{\alpha,\beta}, \eta^{\alpha,\beta}),$$ and $$ S^{\alpha_{0},\beta_{0}}(t)\xi=(u^{\alpha_{0},\beta_{0}}, u_{t}^{\alpha_{0},\beta_{0}}, \eta^{\alpha_{0},\beta_{0}}).$$ 	Then,  $w=u^{\alpha,\beta}-u^{\alpha_{0},\beta_{0}}, \zeta=\eta^{\alpha,\beta}-\eta^{\alpha_{0},\beta_{0}}$ solving
		\begin{equation}\label{6.2}
			\left\{
			\begin{array}{l}
			w_{tt}+kA^{2}w+A^{\alpha}w_{t}+\beta\displaystyle\int_{0}^{\infty}g(s)A^{2}\zeta(x,t,s) ds+(A^{\alpha}-A^{\alpha_{0}})u_{t}^{\alpha_{0},\beta_{0}}\\
			~~~+(\beta-\beta_{0})\int_{0}^{\infty}g(s)A^{2}\eta^{\alpha_{0},\beta_{0}}(x,t,s) ds+f(u^{\alpha,\beta})-f(u^{\alpha_{0},\beta_{0}})=0,\\
				\zeta_{t}=-\zeta_{s}+w_{t}, \\
				(w,w_{t}, \zeta)|_{t=0}=0.
			\end{array}
			\right.
		\end{equation}
	We apply the multiplier $A^{\sigma_{1}}w_{t} + \varepsilon A^{\sigma_{2}}w$ to equation \eqref{6.2}, together with the auxiliary functional associated with the memory term. This yields
		\begin{align}\label{6.3}
			&\frac{d}{dt}	F_{\sigma_{2},\sigma_{2}}(t)
			+\|w_{t}\|_{V_{\sigma_{1}+\alpha}}^{2}+k\varepsilon\|w\|_{V_{2+\sigma_{2}}}^{2}+\frac{\beta\varepsilon}{2}\|\zeta\|_{\mathcal{M}_{2+\sigma_{2}}}^{2}\nonumber\\
			\le&\varepsilon\|w_{t}\|_{V_{\sigma_{2}}}^{2}-\int_{\Omega}\left(f(u^{\alpha,\beta})-f(u^{\alpha_{0},\beta_{0}})\right)\left(A^{\sigma_{1}}w_{t}+\varepsilon A^{\sigma_{2}}w\right)dx\nonumber\\
			&-(\beta-\beta_{0})\int_{0}^{\infty}g(s)\int_{\Omega}A^{2}\eta^{\alpha_{0},\beta_{0}}(x,t,s)\left(A^{\sigma_{1}}w_{t}+\varepsilon A^{\sigma_{2}}w\right)dx ds\nonumber\\
			&-\int_{\Omega}(A^{\alpha}-A^{\alpha_{0}})u_{t}^{\alpha_{0},\beta_{0}}\left(A^{\sigma_{1}}w_{t}+\varepsilon A^{\sigma_{2}}w\right)dx.
		\end{align}
	Here
		\begin{align}\label{6.6}
			F_{\sigma_{1},\sigma_{2}}(t)=& \frac{1}{2}\|(w,w_{t},\zeta)\|_{\mathcal{H}_{\sigma_{1}}}^{2}+\varepsilon \int_{\Omega}w_{t}A^{\sigma_{2}}wdx+\frac{\varepsilon}{2}\|w\|_{V_{\alpha+\sigma_{2}}}^{2}\nonumber\\
			&+\frac{\beta\varepsilon}{2}\int_{0}^{\infty}G(s)\|\zeta-w\|_{V_{\sigma_{2}+2}}^{2}ds.
		\end{align}
		Taking $\sigma_{1}=\alpha$ and  $\sigma_{2}=\min\{2\alpha,2\}$ in \eqref{6.3}, we have, for $\varepsilon$ small enough,
		\begin{align}\label{6.11}
			F_{1}(t)=	F_{\alpha,\sigma_{2}}(t)\ge&\frac{1}{4}\|(w,w_{t},\zeta)\|_{\mathcal{H}_{\alpha}}^{2}.
		\end{align}
		Notice $1\le p<p_{\alpha_{0}}$ for $n>4$ and $1\le p<\infty$ for $1\le n\le4$, we deduce, for $0<\varepsilon_{0}\ll 1$,
		\begin{align}\label{6.7}
			&-\int_{\Omega}\left(f(u^{\alpha,\beta})-f(u^{\alpha_{0},\beta_{0}})\right)\left(A^{\alpha}w_{t}+\varepsilon A^{\sigma_{2}}w\right)dx\nonumber\\
			\le& C\int_{\Omega}\left(1+|u^{\alpha,\beta}|^{p-1}+|u^{\alpha_{0},\beta_{0}}|^{p-1}\right)|w|\left(\left|A^{\alpha}w_{t}\right|+\varepsilon \left|A^{\sigma_{2}}w\right|\right)dx\nonumber\\
			\le&C\|w\|_{L^{\frac{2n}{n-2(2+\sigma_{2}-\varepsilon_{0})}}}\left(\left\|A^{\alpha}w_{t}\right\|+\varepsilon\|A^{\sigma_{2}}w\|\right)\left(1+\|u^{\alpha,\beta}\|_{L^{\frac{(p-1)n}{2+\sigma_{2}-\varepsilon_{0}}}}^{p-1}+\|u^{\alpha_{0},\beta_{0}}\|_{L^{\frac{(p-1)n}{2+\sigma_{2}-\varepsilon_{0}}}}^{p-1}\right)\nonumber\\
			\le&\frac{k\varepsilon}{4}\|w\|_{V_{2+\sigma_{2}}}^{2}+\frac{1}{4}\|w_{t}\|_{V_{2\alpha}}^{2}+C(\varepsilon)\|w\|^{2}.
		\end{align}
		Since $\alpha,\alpha_{0}\in(a,b)$ and $u_{t}^{\alpha_{0},\beta_{0}}\in V_{4-\delta_{0}}$, it follows that
		\begin{align}\label{6.5}
			& -\int_{\Omega}(A^{\alpha}-A^{\alpha_{0}})u_{t}^{\alpha_{0},\beta_{0}}\left(A^{\alpha}w_{t}+\varepsilon A^{\sigma_{2}}w\right)dx\nonumber\\
			\le& \frac{k\varepsilon}{4}\|w\|_{V_{2+\sigma_{2}}}^{2}+\frac{1}{4}\|w_{t}\|_{V_{2\alpha}}^{2}+C(\varepsilon)\left\|\left(A^{\alpha}
			-A^{\alpha_{0}}\right)u_{t}^{\alpha_{0},\beta_{0}}\right\|^{2}<\infty.
		\end{align}
		  Since $\eta^{\alpha_{0},\beta_{0}}\in \mathcal{M}_{4}$, we have
		\begin{align}\label{3.3}
			&(\beta-\beta_{0})\int_{0}^{\infty}g(s)\int_{\Omega}A^{2}  \eta^{\alpha_{0},\beta_{0}} \left(A^{\alpha}w_{t}+\varepsilon A^{\sigma_{2}}w\right)dxds\nonumber\\
			\le&|\beta-\beta_{0}|\int_{0}^{\infty}g(s)\int_{\Omega}|A^{2}\eta^{\alpha_{0},\beta_{0}}|\left(|A^{\alpha}w_{t}|+\varepsilon |A^{\sigma_{2}}w|\right)dxds\nonumber\\
			\le&C|\beta-\beta_{0}|\|\eta^{\alpha_{0},\beta_{0}}\|_{\mathcal{M}_{4}}\left(\|w_{t}\|_{V_{2\alpha}}+\varepsilon \|w\|_{V_{2\sigma_{2}}}\right)\nonumber\\
			\le&\frac{1}{8}\|w_{t}\|_{V_{2\alpha}}^{2}+\frac{k\varepsilon}{4}\|w\|_{V_{2+\sigma_{2}}}^{2}+C|\beta-\beta_{0}|^{2}.
		\end{align}
		Inserting \eqref{6.7}, \eqref{6.5} and \eqref{3.3} into \eqref{6.3}, taking $\varepsilon$ small enough, we get
		\begin{align}\label{6.9}
			&\frac{d}{dt}F_{1}(t)+\frac{1}{4}\|w_{t}\|_{V_{2\alpha}}^{2}+\frac{k\varepsilon}{4}\|w\|_{V_{2+\sigma_{2}}}^{2}+\frac{\beta\varepsilon}{2}\|\eta\|_{\mathcal{M}_{2+\sigma_{2}}}^{2}\nonumber\\
			\le &C\|w\|^{2}+C\|(A^{\alpha}-A^{\alpha_{0}})u_{t}^{\alpha_{0},\beta_{0}}\|^{2}+C|\beta-\beta_{0}|^{2}.
		\end{align}
		Applying the Gronwall inequality to \eqref{6.9} yields
		\begin{align}\label{6.13}
			F_{1}(t)\le C\int_{0}^{t}e^{c(t-\tau)}\|(A^{\alpha}-A^{\alpha_{0}})u_{t}^{\alpha_{0},\beta_{0}}\|^{2}d\tau+Ce^{ct}|\beta-\beta_{0}|^{2}
		\end{align}
		for some constant $c>0$. Since $$(u_{0}^{\alpha_{0},\beta_{0}}, u_{1}^{\alpha_{0},\beta_{0}}, \eta_{0}^{\alpha_{0},\beta_{0}})\in \mathcal{B}\subset \mathcal{Z},$$ we define $$u_{t}^{\alpha_{0},\beta_{0}}(t)=\sum_{j=1}^{\infty}a_{j}(t)\omega_{j},$$ where $A\omega_{j}=\lambda_{j}\omega_{j}, j=1,2,3,...$, and from \eqref{6.5}, for any $\varepsilon>0$ there exists a positive integer $N>0$ such that
		\begin{align*}
			\sum_{j>N}\left| \left(\lambda_{j}^{\alpha}-\lambda_{j}^{\alpha_{0}}\right)a_{j}(t)\right|^{2}\le \frac{\varepsilon}{2},
		\end{align*}
		and there exists a $0<\widehat{\delta}<\delta_{1}$ such that
		\begin{align*}
			\sum_{j=1}^{N}\left| \left(\lambda_{j}^{\alpha}-\lambda_{j}^{\alpha_{0}}\right)a_{j}(t)\right|^{2}\le \frac{\varepsilon}{2}, \mbox{~as~}0<|\alpha-\alpha_{0}|<\widehat{\delta}.
		\end{align*}
		Therefore, we have
		\begin{align}\label{6.27}
			\lim_{\alpha\to\alpha_{0}} \|(A^{\alpha}-A^{\alpha_{0}})u_{t}^{\alpha_{0},\beta_{0}}\|^{2} =0.
		\end{align}
		By \eqref{6.5}  and the Lebesgue dominated convergence theorem, we have
		\begin{align}\label{6.12}
			\lim_{\alpha\to\alpha_{0}} \int_{0}^{t}e^{c(t-\tau)}\|(A^{\alpha}-A^{\alpha_{0}})u_{t}^{\alpha_{0},\beta_{0}}\|^{2}d\tau=0.
		\end{align}
		Combining \eqref{6.11} and \eqref{6.12}, and by \eqref{6.6} and \eqref{6.13}, we infer
		\begin{align}\label{6.31}
			\lim_{(\alpha,\beta)\to(\alpha_{0},\beta_{0})}\|(w,w_{t}, \zeta)\|_{\mathcal{H}_{\alpha}}^{2}=0.
		\end{align}
		Furthermore, by the interpolation theorem, from \eqref{6.31}, we get
		
	\begin{align}\label{6.24}
			\|w_{t}\|_{V_{4-\epsilon}}\le \|w_{t}\|_{V_{4-\delta_{0}}}^{\frac{4-\alpha-\epsilon}{4-\alpha-\delta_{0}}}\|w_{t}\|_{V_{\alpha}}^{\frac{\epsilon-\delta_{0}}{4-\alpha-\delta_{0}}}\le C(\mathcal{B})\|w_{t}\|_{V_{\alpha}}^{\frac{\epsilon-\delta_{0}}{4-\alpha-\delta_{0}}}\to 0, \mbox{~as~}(\alpha,\beta)\to(\alpha_{0},\beta_{0}),
		\end{align}
		and
		\begin{align}\label{6.24-}
		\|(w,w_{t}, \zeta)\|_{\mathcal{H}_{2-\epsilon}}\le C\|(w,w_{t}, \zeta)\|_{\mathcal{H}_{2}}^{\frac{2-\epsilon}{2}}\|(w,w_{t}, \zeta)\|_{\mathcal{H}}^{\frac{\epsilon}{2}}\to 0, \mbox{~as~}(\alpha,\beta)\to(\alpha_{0},\beta_{0}),
		\end{align}
where $\epsilon$ is defined in \eqref{mu}.	Therefore, combining \eqref{6.24} and \eqref{6.24-}, we obtain
	\begin{align}\label{key0}
	\lim_{(\alpha,\beta)\to(\alpha_{0},\beta_{0})}	\|S^{\alpha,\beta}(t)\xi-S^{\alpha_{0},\beta_{0}}(t)\xi\|_{\mathcal{Z}_{\epsilon}}=0.
	\end{align}
We note that $\mathcal{B}_{R}$ is a complete metric space endowed with the norm derived from $\mathcal{Z}_{\epsilon}$, given that it is a closed subset of $\mathcal{Z}_{\epsilon}$. Consequently, the strong $(\mathcal{H},\mathcal{Z}_{\epsilon})$-global attractor $\mathcal{A}_{\alpha,\beta}$ emerges as the global attractor of the dynamical system $\left(\mathcal{B}_{R}, S^{\alpha,\beta}(t)\right)$.
Furthermore, by invoking Lemma \ref{lemma6.1} , we arrive at \eqref{3.3-}.
	
	\textbf{(iii)}	In particular, when $\beta \in [0,1]$ is fixed and $\alpha_{0}\in\lbrack0,2)$, the third term on the right-hand side of \eqref{6.3} vanishes. For brevity, we denote
$$A_{\alpha} = A_{\alpha,\beta}, \quad u^{\alpha} = u^{\alpha,\beta}, \quad u_{t}^{\alpha} = u_{t}^{\alpha,\beta}, \quad S^{\alpha}(t) = S^{\alpha,\beta}(t).$$ Taking $\sigma_{1} = \sigma_{2} = 2$, we obtain
		\begin{equation}\label{7.1}
			F_{2}(t) = F_{2,2}(t) \sim \|(w, w_{t}, \zeta)\|_{\mathcal{H}_{2}}^{2}.
		\end{equation}
		Consequently, we have the following estimate:
		\begin{align}\label{6.7-}
			&-\int_{\Omega}\left(f(u^{\alpha})-f(u^{\alpha_{0}})\right)\left(A^{2}w_{t}+\varepsilon A^{2}w\right)dx\nonumber\\
			=&-\int_{\Omega}\left(f^{\prime}(u^{\alpha})A^{\frac{1}{2}}u^{\alpha}-f^{\prime}(u^{\alpha_{0}})A^{\frac{1}{2}}u^{\alpha_{0}}\right)A^{\frac{3}{2}}w_{t}dx\nonumber\\
			&-\varepsilon\int_{\Omega}\left(f(u^{\alpha})-f(u^{\alpha_{0}})\right)A^{2}wdx\nonumber\\
			\le& C\|A^{\frac{3}{2}}w_{t}\|_{L^{\frac{2n}{n-2\left(1-\epsilon\right)}}}\left(\|A^{\frac{1}{2}}u^{\alpha}\|_{\frac{2n}{n-6}}+\|A^{\frac{1}{2}}u^{\alpha_{0}}\|_{\frac{2n}{n-6}}\right)\nonumber\\
			&\times\left(1+\|u^{\alpha}\|_{L^{\frac{(p-1)n}{4-\epsilon}}}^{p-1}+\|u^{\alpha_{0}}\|_{L^{\frac{(p-1)n}{4-\epsilon}}}^{p-1}\right)\nonumber\\
			&+C\varepsilon\|w\|_{L^{\frac{2n}{n-2(4-\varepsilon_{0})}}}\|w\|_{V_{4}}\left(1+\|u^{\alpha}\|_{L^{\frac{(p-1)n}{4-\varepsilon_{0}}}}+\|u^{\alpha_{0}}\|_{L^{\frac{(p-1)n}{4-\varepsilon_{0}}}}\right)\nonumber\\
			\le&\frac{k\varepsilon}{4}\|w\|_{V_{4}}^{2}+C\left(\|w\|^{2}+\|w_{t}\|_{V_{4-\epsilon}}\right).
		\end{align}
		Since $\alpha, \alpha_{0}\in (a,b) $ and $u_{t}^{\alpha_{0}}\in V_{4-\delta_{0}}$, we get
		\begin{align}\label{6.5-}
			& -\int_{\Omega}(A^{\alpha}-A^{\alpha_{0}})u_{t}^{\alpha_{0}}\left(A^{2}w_{t}+\varepsilon A^{2}w\right)dx\nonumber\\
			=&-\int_{\Omega}(A^{\alpha+\frac{\epsilon}{2}}-A^{\alpha_{0}+\frac{\epsilon}{2}})u_{t}^{\alpha_{0}}\left(A^{2-\frac{\epsilon}{2}}w_{t}+\varepsilon A^{2-\frac{\epsilon}{2}}w\right)dx\nonumber\\
			\le& \frac{k\varepsilon}{4}\|w\|_{V_{4}}^{2}+\frac{1}{4}\|w_{t}\|_{V_{4-\epsilon}}^{2}+C(\varepsilon)\left\|\left(A^{\alpha
				+\frac{\epsilon}{2}}-A^{\alpha_{0}+\frac{\epsilon}{2}}\right)u_{t}^{\alpha_{0}}\right\|^{2}<\infty.
		\end{align}
		Substituting estimates  \eqref{6.7-} and \eqref{6.5-} into \eqref{6.3},  we deduce that
		\begin{align*}
			&\frac{d}{dt}F_{2}(t)+\|w_{t}\|_{V_{2+\alpha}}^{2}+\frac{k\varepsilon}{2}\|w\|_{V_{4}}^{2}+\frac{\beta\varepsilon}{2}\|\zeta\|_{\mathcal{M}_{4}}^{2}\nonumber\\
			\le& C\left(\|w\|^{2}+\|w_{t}\|_{V_{4-\epsilon}}^{2}+\|w_{t}\|_{V_{4-\epsilon}}+\left\|\left(A^{\alpha
				+\frac{\epsilon}{2}}-A^{\alpha_{0}+\frac{\epsilon}{2}}\right)u_{t}^{\alpha_{0}}\right\|^{2}\right).
		\end{align*}
		By Gronwall's inequality and \eqref{7.1} , there exists a constant $\gamma_{1}>0$ such that
		\begin{align}\label{7.2-}
			\|(w,w_{t},\zeta)\|_{\mathcal{H}_{2}}^{2}\le& C\int_{0}^{t}e^{-\gamma_{1}(t-s)}\left(\|w\|^{2}+\|w_{t}\|_{V_{4-\epsilon}}^{2}+\|w_{t}\|_{V_{4-\epsilon}}\right)ds\nonumber\\
			&+C\int_{0}^{t}e^{-\gamma_{1}(t-s)}\left\|\left(A^{\alpha
				+\frac{\epsilon}{2}}-A^{\alpha_{0}+\frac{\epsilon}{2}}\right)u_{t}^{\alpha_{0}}\right\|^{2}ds.
		\end{align}
		Similarly as in the proof of \eqref{6.27}, we also have
		\begin{align}\label{7.3-}
			\lim_{\alpha\to\alpha_{0}} \left\|\left(A^{\alpha
				+\frac{\epsilon}{2}}-A^{\alpha_{0}+\frac{\epsilon}{2}}\right)u_{t}^{\alpha_{0}}\right\|^{2} =0.
		\end{align}
		By employing the inequalities \eqref{6.31}, \eqref{6.24}, and \eqref{7.3-}, along with the Lebesgue dominated convergence theorem, it follows from the estimate \eqref{7.2-}  that
		\begin{align}\label{r}
			\lim_{\alpha\to\alpha_{0}}\|S^{\alpha,\beta}(t)\xi-S^{\alpha_{0},\beta}(t)\xi\|_{V_{4-\epsilon} \times V_{4} \times \mathcal{M}_{4}}=0.
		\end{align}
 Then, from Lemma \ref{lemma6.1}, we obtain \eqref{3.4}, which completes the proof.	
\end{proof}

\begin{remark}{\rm
	From the proof of Theorem \ref{them5}, it follows that, for each fixed $\beta \in [0,1]$, the attractor is upper semicontinuous with respect to $\alpha\in\lbrack0,2)$ in the topology of $V_{4} \times V_{4-\epsilon} \times \mathcal{M}_{4}$. However, the upper semicontinuity with respect to $\beta$ can only be established in the weaker $\mathcal{Z}_{\varepsilon}$-topology. The reason is that the term
	\begin{align*}
		(\beta - \beta_{0}) \int_{0}^{\infty} g(s) \int_{\Omega} A^{2} \eta^{\alpha_{0},\beta_{0}} A^{2} w_{t} \, dx \, ds
	\end{align*}
	is difficult to control when the multiplier $A^{2}w_{t}$ is applied.

We also point out that the upper semicontinuity of the attractor
family with respect to $\alpha$ at $\alpha_0=2$ would require taking
$\epsilon=0$ in \eqref{6.5-} for fixed $\beta\in[0,1]$. This produces the term $
	\|w_t\|_{V_4}^{2}.$
	However, for \(\alpha<2\), the dissipation term only controls
$	\|w_t\|_{V_{2+\alpha}}^{2},$
	which is not strong enough to absorb $\|w_t\|_{V_4}^{2}$. Therefore,
	the present argument does not yield the upper semicontinuity at
	$\alpha_0=2$ in the stronger topology
$V_4\times V_{4-\varepsilon}\times\mathcal M_4 .$}
\end{remark}

\section{Continuity of strong exponential attractors}

 The purpose of this section is to demonstrate that the dynamical system
$(\mathcal{H},S^{\alpha,\beta}(t))$ possesses a family of strong $(\mathcal{H},\mathcal{Z}_{\epsilon})$-exponential attractors $\mathcal{E}_{\exp}^{\alpha,\beta}$, which depend continuously on
$(\alpha,\beta)\in\Xi$.

To achieve this goal, we rely on the following abstract criterion established in \cite{Yang-2018}.
\begin{lemma}\label{lem Yang}
 Let  $M$ be a bounded closed set of the Banach space $X$. Assume that for every $\chi\in I$ (index set) there exists a family of discrete semigroup  $\{V_{\chi}^{k}\}_{k\in \mathbb{N}}$ acting on $M$, and  possessing the following properties:
	\begin{enumerate}
		\item[{\rm(i)}] $V_{\chi}^{k}$ is uniformly (w.r.t. $\chi\in I$) Lipschitz continuous on $M$, i.e., there exists a positive constant $L$ such that
		\begin{align*}
		\sup_{\chi\in I}	\|V_{\chi}^{k}v_{1}-V_{\chi}^{k}v_{2}\|_{X}\le L\|v_{1}-v_{2}\|_{X}, ~\forall v_{1}, v_{2}\in M.
		\end{align*}
		\item[{\rm(ii)}]there exist a Banach space $Z$ and a compact seminorms $n_{Z}(\cdot)$ on $Z$, and there exists a mapping $K_{\chi}: M\to Z$ for each $\chi\in I$ such that
		\begin{align*}
			&\sup_{\chi\in I}\|K_{\chi}v_{1}-K_{\chi}v_{2}\|_{Z}\le L_{1}\|v_{1}-v_{2}\|_{X}, \\
			& \|V_{\chi}^{k}v_{1}-V_{\chi}^{k}v_{2}\|_{X}\le \theta\|v_{1}-v_{2}\|_{X}+n_{Z}(K_{\chi}v_{1}-K_{\chi}v_{2}), ~\forall v_{1},v_{2}\in M,
		\end{align*}
		where $0<\theta<1, L_{1}>0$ are constants independent of $\chi$.
	\end{enumerate}
		Then, the discrete dynamical system $(V_{\chi}^{k},M)$ admits an exponential attractor $\Sigma^{\chi}$ for each $\chi\in I$. Moreover, for any given $\chi_{0}\in I$, if $\chi\in I$ satisfies
		\begin{align*}
			\Gamma(\chi,\chi_{0}):=\sup_{x\in M}\|V_{\chi}^{k}x-V_{\chi_{0}}^{k}x\|_{X}< 1,
		\end{align*}
		then,
		\begin{align*}
			\mbox{dist}_{X}^{symm}\left(\Sigma^{\chi}, \Sigma^{\chi_{0}}\right)\le C\left(\Gamma(\chi,\chi_{0})\right)^{\lambda}
		\end{align*}
		where $C>0$ and $0<\lambda<1$ are constants independent of $\chi$.
	
\end{lemma}

We next present the quasi-stability estimate for the system $(\mathcal{H},S^{\alpha,\beta}(t))$ .
	\begin{theorem}\label{thmm}
Assume that Assumption~\ref{ass1} holds, and that
\begin{equation}\label{dim}
	\left\{
	\begin{array}{ll}
		1\le n\le
	\dfrac{8\min\{1+\alpha,2\}}
	{\min\{\alpha,2-\alpha\}}+4,~&\mbox{if}~ 0<\alpha<2~\mbox{and}~0<\beta\le1,\\
	1\le n<+\infty,\quad\quad\quad\quad\quad\quad\quad&\mbox{if}~ \alpha=0,2~\mbox{or}~\beta=0.
	\end{array}
	\right.
\end{equation}
Then, for any  full trajectory $$\left\{S^{\alpha,\beta}(t)z^{i}=(u^{i},u_{t}^{i},\eta^{i})| z^{i}=(u_{0}^{i}, u_{1}^{i}, \eta_{0}^{i})\in \mathcal{B}, i=1,2,t \in \mathbb{R}\right\},$$ we have the estimates
	\begin{align}\label{7.2}
		\|S^{\alpha,\beta}(t)z^{1}-S^{\alpha,\beta}(t)z^{2}\|_{\mathcal{H}_{\sigma}}^{2}\le C_{0}e^{\gamma_{2}t}\|z^{1}-z^{2}\|_{\mathcal{H}_{\sigma}}^{2}
	\end{align}
	and the quasi-stability inequality
	\begin{align}\label{7.3}
		\|S^{\alpha,\beta}(t)z^{1}-S^{\alpha,\beta}(t)z^{2}\|_{\mathcal{H}_{\sigma}}^{2}\le C_{0}e^{-\gamma_{3}t}\|z^{1}-z^{2}\|_{\mathcal{H}_{\sigma}}^{2}+C_{0}\left(1-e^{-\gamma_{3}t}\right)\sup_{s\in[0,t]}\|u^{1}-u^{2}\|^{2},
	\end{align}
for some constants $\gamma_{2}, \gamma_{3}, C_{0} > 0$, and for any $0\le\sigma\le \alpha$.	Then, the system $(\mathcal{H},S^{\alpha,\beta}(t))$ has an exponential attractor $E_{exp}^{\alpha,\beta}$ whose dimension is finite on $\mathcal{H}$.
\end{theorem}

\begin{proof}
	For any two trajectories
	$$\left\{\xi_{u^{i}}|\xi_{u^{i}}=S^{\alpha,\beta}(t)z^{i}=(u^{i},u_{t}^{i},\eta^{i}),~z^{i}\in \mathcal{B},~i=1,2\right\},$$
	set $\xi_{\phi}=(\phi,\phi_{t},\psi)=\xi_{u^{1}}-\xi_{u^{2}}$, which satisfies the following problem
	\begin{equation*}
		\left\{
		\begin{array}{l}
			\phi_{tt}+kA^{2}\phi+A^{\alpha}\phi_{t}+\beta\displaystyle\int_{0}^{\infty}g(s)A^{2}  \psi ds+f(u^{1})-f(u^{2})=0,\\
			\psi_{t}=-\psi_{s}+\phi_{t},\\
			(\phi,\phi_{t}, \psi)|_{t=0}=z^{1}-z^{2}.
		\end{array}
		\right.
	\end{equation*}

When $0<\beta \le 1$, for $1\le p<p^{**}<\frac{n+2\alpha}{(n-4)^{+}}$ we introduce the following auxiliary functional
\begin{align*}
H(t)=\frac{1}{2}\|(\phi,\phi_{t},\psi)\|_{\mathcal{H}_{\sigma}}^{2}+\varepsilon \int_{\Omega}\phi_{t}A^{\sigma}\phi dx+\frac{\varepsilon}{2}\|\phi\|_{V_{\alpha+\sigma}}^{2}+\frac{\beta\varepsilon}{2}\int_{0}^{\infty}G(s)\|\psi-\phi\|_{V_{2+\sigma}}^{2}ds,
\end{align*}
with $0\le \sigma\le \alpha$. Differentiating with respect to $t$ yields, for any $0<\varepsilon_{0}\ll 1$,
	\begin{align}\label{ff}
&\frac{d}{dt}H(t)+\|\phi_{t}\|_{V_{\alpha+\sigma}}^{2}+k\varepsilon\|\phi\|_{V_{2+\sigma}}^{2}-\varepsilon\|\phi_{t}\|_{V_{\sigma}}^{2}+\frac{\beta\varepsilon}{2}\|\psi\|_{\mathcal{M}_{2+\sigma}}^{2}\nonumber\\
\le&-\int_{\Omega}(f(u^{1})-f(u^{2}))(A^{\sigma}\phi_{t}+\varepsilon A^{\sigma}\phi)dx\nonumber\\
	\le& C\|\phi\|_{L^{\frac{2n}{n-2(2+\sigma-\varepsilon_{0})}}}\left(\|A^{\sigma}\phi_{t}\|_{L^{\frac{2n}{n-2(\alpha-\sigma)}}}+\varepsilon\|A^{\sigma}\phi\|_{L^{\frac{2n}{n-2(2-\sigma)}}}\right)\nonumber\\
&\times \left(1+\|u^{1}\|_{L^{\frac{(p-1)n}{2+\alpha-\varepsilon_{0}}}}^{p-1}+\|u^{2}\|_{L^{\frac{(p-1)n}{2+\alpha-\varepsilon_{0}}}}^{p-1}+\|u^{1}\|_{L^{\frac{(p-1)n}{4-\varepsilon_{0}}}}^{p-1}+\|u^{2}\|_{L^{\frac{(p-1)n}{4-\varepsilon_{0}}}}^{p-1}\right)\nonumber\\
\le& \frac{1}{4}\|\phi_{t}\|_{V_{\alpha+\sigma}}^{2}+\frac{k\varepsilon}{2}\|\phi\|_{V_{2+\sigma}}^{2}+C(\varepsilon)\|\phi\|^{2}.
	\end{align}
	To justify the above estimate uniformly for \(0\le \sigma\le \alpha\)
	and for every \(p^{**}\le p<p^{*}\), we use the higher-order regularity $
	u^{1},u^{2}\in V_{4}.$ For sufficiently small $\varepsilon_0$, the following condition on the dimension $n>8$ must hold:
	\begin{align*}
		\frac{(p-1)n}{2 + \alpha} < \frac{2n}{n - 8} \quad \Longrightarrow \quad p < p^* \leq \frac{n - 4 + 2\alpha}{n - 8}.
	\end{align*}
		On the other hand, if \(1\le n\le 8\), then $
	V_{4}\hookrightarrow L^{\infty}(\Omega),$
	and the above restriction is automatically satisfied. Hence the estimate
	\eqref{ff} is valid under the dimensional condition \eqref{dim}. Then, choosing $\varepsilon$ sufficiently small,  we obtain
\begin{align*}
	C_{1}\|S^{\alpha,\beta}(t)z^{1}-S^{\alpha,\beta}(t)z^{2}\|_{\mathcal{H}_{\sigma}}^{2}\le H(t)\le C_{2} \|S^{\alpha,\beta}(t)z^{1}-S^{\alpha,\beta}(t)z^{2}\|_{\mathcal{H}_{\sigma}}^{2},
\end{align*}
and
	\begin{align*}
	\frac{d}{dt}H(t)+\frac{1}{2}\|\phi_{t}\|_{V_{\alpha+\sigma}}^{2}+\frac{k\varepsilon}{2}\|\phi\|_{V_{2+\sigma}}^{2}+\frac{\beta\varepsilon}{2}\|\psi\|_{\mathcal{M}_{2+\sigma}}^{2}\le C(\varepsilon)\|\phi\|^{2},
	\end{align*}
where, $C_{1},C_{2},C(\varepsilon)$ are positive constants.	By Gronwall's inequality, we then obtain estimates \eqref{7.2} and \eqref{7.3}.

If the memory term vanishes, i.e. $\beta=0$, the critical condition for the nonlinear term is given by \eqref{assp}. We construct the following auxiliary functional
	\begin{align*}
		H_{1}(t)=\frac{1}{2}(\|\phi_{t}\|_{V_{\sigma}}^{2}+k\|\phi\|_{V_{2+\sigma}}^{2})+\varepsilon \int_{\Omega}\phi_{t}A^{\widetilde{\sigma}}\phi dx+\frac{\varepsilon}{2}\|\phi\|_{V_{\alpha+\widetilde{\sigma}}}^{2},
	\end{align*}
where $\widetilde{\sigma}=\min\{\alpha,2-\alpha\}+\sigma$ and $0 \le \sigma \le\alpha$.  Differentiating $H_{1}(t)$ with respect to $t$, we obtain the following estimate:
	\begin{align*}
	&\frac{d}{dt}H_{1}(t)+\|\phi_{t}\|_{V_{\alpha+\sigma}}^{2}+k\varepsilon\|\phi\|_{V_{2+\widetilde{\sigma}}}^{2}-\varepsilon\|\phi_{t}\|_{V_{\widetilde{\sigma}}}^{2}\nonumber\\
	\le&-\int_{\Omega}(f(u^{1})-f(u^{2}))(A^{\sigma}\phi_{t}+\varepsilon A^{\widetilde{\sigma}}\phi)dx\nonumber\\
	\le& C\|\phi\|_{L^{\frac{2n}{n-2(2+\widetilde{\sigma}-\varepsilon_{0})}}}\left(\|A^{\sigma}\phi_{t}\|_{L^{\frac{2n}{n-2(\alpha-\sigma)}}}+\varepsilon\|A^{\widetilde{\sigma}}\phi\|_{L^{\frac{2n}{n-2(2-\widetilde{\sigma})}}}\right)\nonumber\\
	&\times \left(1+\|u^{1}\|_{L^{\frac{(p-1)n}{2+2\alpha-\varepsilon_{0}}}}^{p-1}+\|u^{2}\|_{L^{\frac{(p-1)n}{2+2\alpha-\varepsilon_{0}}}}^{p-1}+\|u^{1}\|_{L^{\frac{(p-1)n}{4-\varepsilon_{0}}}}^{p-1}+\|u^{2}\|_{L^{\frac{(p-1)n}{4-\varepsilon_{0}}}}^{p-1}\right)\nonumber\\
	\le& \frac{1}{4}\|\phi_{t}\|_{V_{\alpha+\sigma}}^{2}+\frac{k\varepsilon}{2}\|\phi\|_{V_{2+\widetilde{\sigma}}}^{2}+C(\varepsilon)\|\phi\|^{2}.
\end{align*}
Taking $\varepsilon > 0$ sufficiently small, we have
\begin{align*}
	\frac{d}{dt}H_{1}(t)+\frac{1}{2}\|\phi_{t}\|_{V_{\alpha+\sigma}}^{2}+\frac{k\varepsilon}{2}\|\phi\|_{V_{2+\widetilde{\sigma}}}^{2}\le C(\varepsilon)\|\phi\|^{2},
\end{align*}
and for some positive constants $C_{3}, C_{4} > 0$,
\begin{align*}
	C_{3}\|S^{\alpha,0}(t)z^{1}-S^{\alpha,0}(t)z^{2}\|_{\mathcal{H}_{\sigma}}^{2}\le H_{1}(t)\le C_{4} \|S^{\alpha,0}(t)z^{1}-S^{\alpha,0}(t)z^{2}\|_{\mathcal{H}_{\sigma}}^{2}.
\end{align*}
 Here $\mathcal{H}_{\sigma} = V_{2+\sigma} \times V_{\sigma} \times \{0\}$. Applying Gronwall's inequality yields the estimates \eqref{7.2} and \eqref{7.3} for sufficiently small $\varepsilon > 0$, with no dimensional restrictions imposed.	

Based on the foregoing discussion, estimates \eqref{7.2} and \eqref{7.3} yield the quasi-stability of the system and the following estimate
\begin{align*}
\|\phi_{t}\|_{V_{2+\sigma}}^{2}+\|\phi_{tt}\|_{V_{\sigma}}^{2}+\beta\|\psi_{t}\|_{\mathcal{M}_{2+\sigma}}^{2}\le C.
\end{align*}
  For that, for any $t_{1}, t_{2}\in[0,T], z\in \mathcal{B},$ we obtain
\begin{align}\label{key}
	\|S^{\alpha,\beta}(t_{1})z-S^{\alpha,\beta}(t_{2})z\|_{\mathcal{H}_{\alpha}}\le \left(\int_{t_{1}}^{t_{2}}\left\|\frac{d}{dt}S^{\alpha,\beta}(t)z\right\|_{\mathcal{H}_{\alpha}}^{2}dt\right)^{\frac{1}{2}}\le C|t_{1}-t_{2}|^{\frac{1}{2}}.
\end{align}
Thus,  the proof of Theorem \ref{thmm} is proved.

\end{proof}

\begin{theorem}\label{them4.2}
Let Assumption in Theorem \ref{thmm} be satisfied. Then, the dynamical system $(\mathcal{H}, S^{\alpha,\beta}(t))$ possesses a strong $(\mathcal{H}, \mathcal{Z}_{\epsilon} )$-exponential attractor $\mathcal{E}_{exp}^{\alpha,\beta}$ for each $(\alpha,\beta)\in \Xi$. Furthermore, the family $\mathcal{E}_{\exp}^{\alpha,\beta}$ exhibits the following continuity at $(\alpha_0 ,\beta_{0})\in\Xi$:
\begin{align}\label{4.3=}
	\lim_{(\alpha,\beta)\to(\alpha_{0},\beta_{0})}\mbox{dist}_{V_{4-\epsilon} \times V_{4-\epsilon}\times M_{4-\epsilon}} ^{symm}\left(\mathcal{E}_{exp}^{\alpha,\beta}, \mathcal{E}_{exp}^{\alpha_{0},\beta_{0}}\right)=0.
\end{align}
 Moreover, for $\beta \equiv 0$, the exponential attractor is continuous at $\alpha_{0}\in\lbrack0,2)$, in a stronger topology:
\begin{equation}\label{4.33}
	\lim_{\alpha \to \alpha_{0}}
	\operatorname{dist}_{V_{4} \times V_{4-\epsilon}}
	(\mathcal{E}_{\mathrm{exp}}^{\alpha,0},\, \mathcal{E}_{\mathrm{exp}}^{\alpha_{0},0}) = 0.
\end{equation}

\end{theorem}
\begin{proof}
 We define the functional space $$Z=\left\{\xi_u| \ \xi_{u}=(u, u_{t}, \eta)\in C([0,T]; \mathcal{H})\right\}$$ equipped with norm $$\|\xi_{u}\|_{Z}=\sup_{t\in[0,T]}\|(u,u_{t},\eta)\|_{\mathcal{H}},$$ which is a Banach space. Define the  operators
    \begin{align*}
        V_{\alpha,\beta}^{k}=S^{\alpha,\beta}(kT): \mathcal{B}\to \mathcal{B}
    \end{align*}
 with  $k\in N^{+},~T\ge \max\{2\gamma_{3}^{-1}\ln(\sqrt{C_{0}}+1)+1, t_{B_{0}}+1\},$ and
    \begin{align*}
    K_{\alpha,\beta}:\mathcal{B}\to Z,~ K_{\alpha,\beta}\xi=\xi_{u^{\alpha,\beta}},~ \forall \xi\in \mathcal{B},
    \end{align*}
    where $t_{B_{0}}$ and  $\mathcal{B}$ are defined in \eqref{B}, and $\xi_{u^{\alpha,\beta}}=S^{\alpha,\beta}(t)\xi=(u^{\alpha,\beta}, u_{t}^{\alpha,\beta}, \eta^{\alpha,\beta})$; also define
    \begin{align*}
        n_{Z}: Z\to R, ~n_{Z}(\xi_{u})=C(\mathcal{B})\sup_{t\in[0,T]}\|u\|, ~\forall \xi_{u}\in Z.
    \end{align*}
 Then, $n_{Z}$ is a compact seminorm on $Z$. It follows from \eqref{7.2} and \eqref{7.3} that, for any $\xi^{1},\xi^{2}\in \mathcal{B}$, taking $ \theta=\sqrt{C_{0}}e^{-\frac{1}{2}\gamma_{3}T}<1$, we get
\begin{align*}
    \sup_{(\alpha,\beta)\in\Xi}\|K_{\alpha,\beta}\xi^{1}-K_{\alpha,\beta}\xi^{2}\|_{Z}\le Ce^{\gamma_{2}T}\|\xi^{1}-\xi^{2}\|_{\mathcal{H}},
\end{align*}
\begin{align*}
\sup_{(\alpha,\beta)\in\Xi}\|V_{\alpha,\beta}^{k}\xi^{1}-V_{\alpha,\beta}^{k}\xi^{2}\|_{\mathcal{H}}\le Ce^{k\gamma_{2}T}\|\xi^{1}-\xi^{2}\|_{\mathcal{H}},
\end{align*}
\begin{align*}
\|V_{\alpha,\beta}^{k}\xi^{1}-V_{\alpha,\beta}^{k}\xi^{2}\|_{\mathcal{H}}\le \theta \|\xi^{1}-\xi^{2}\|_{\mathcal{H}}+n_{Z}(K_{\alpha,\beta}\xi^{1}-K_{\alpha,\beta}\xi^{2}).
\end{align*}
Therefore, by Lemma \ref{lem Yang}, for each $(\alpha,\beta)\in \Xi$, the discrete dynamical system $(\mathcal{B},V_{\alpha,\beta}^{k})$, whose norm is equivalent to that of $\mathcal{H}$, admits an exponential attractor $M_{\alpha,\beta}$.
Let
\begin{align*}
	\mathcal{A}_{exp}^{\alpha,\beta}=\bigcup_{0\le t\le T}S^{\alpha,\beta}(t)M_{\alpha,\beta},~ (\alpha,\beta)\in \Xi,
\end{align*}
which is a forward invariant compact set in $\mathcal{H}$. By \eqref{7.2} and \eqref{key}, we get
\begin{align*}
	\dim_{f}\mathcal{A}_{exp}^{\alpha,\beta}\le 2 \left(1+\dim_{f}M_{\alpha,\beta}\right)<\infty.
\end{align*}
For any bounded set $D\subset\mathcal{H}$, there exists $t_D>0$ such that $S^{\alpha,\beta}(s)D\subset \mathcal{B}$ for all $s\ge t_D$. Writing $t=kT+t_{D}+\tau$ with $k\in\mathbb{N}$ and $\tau\in[0,T)$, we deduce from \eqref{7.2} the existence of a constant $\kappa>0$ such that
\begin{align}\label{4.11}
	\mbox{dist}_{\mathcal{H}}\left(S^{\alpha,\beta}(t)D, \mathcal{A}_{exp}^{\alpha,\beta}\right)\le & \mbox{dist}_{\mathcal{H}}\left(S^{\alpha,\beta}(\tau)S^{\alpha,\beta}(kT)S^{\alpha,\beta}(t_{D})D, S^{\alpha,\beta}(\tau)M_{\alpha,\beta}\right)\nonumber\\
	\le& Ce^{\gamma_{2}\tau}\mbox{dist}_{\mathcal{H}}\{V_{\alpha,\beta}^{k}\mathcal{B}, \mathcal{M}_{\alpha,\beta}\}\le C_{T}e^{-\kappa (t-t_{D})}.
\end{align}
Here, $C_{T}= Ce^{(\gamma_{2}+\kappa)T}$. Thus, $\mathcal{A}_{exp}^{\alpha,\beta}$ is an exponential attractor of the continuous dynamical system $(\mathcal{H}, S^{\alpha,\beta}(t))$.
Let  $$\mathcal{E}_{exp}^{\alpha,\beta}=S^{\alpha,\beta}(1)\mathcal{A}_{exp}^{\alpha,\beta},~(z,z_{t},\eta)=S^{\alpha,\beta}(t)
\xi_{1}-S^{\alpha,\beta}(t)\xi_{2}, \mbox{~for~any ~} \xi_{1},\xi_{2}\in \mathcal{B}.$$
	For \(\beta>0\) and \(\alpha\in[0,2]\), by the interpolation inequality,
	we have
	\begin{align}\label{4.6}
		&\|S^{\alpha,\beta}(t)\xi_{1}
		-S^{\alpha,\beta}(t)\xi_{2}\|_{ V_{4-\epsilon}\times V_{4-\epsilon}\times \mathcal{M}_{4-\epsilon}}
		\nonumber\\
		&\leq C\Big(
		\|z\|_{V_4}^{\frac{2-\epsilon}{2}}
		\|z\|_{V_2}^{\frac{\epsilon}{2}}
		+
		\|z_t\|_{V_{4-\delta_0}}^{\frac{4-\epsilon}{4-\delta_0}}
		\|z_t\|^{\frac{\epsilon-\delta_0}{4-\delta_0}}+\beta
		\|\eta\|_{\mathcal M_4}^{\frac{2-\epsilon}{2}}
		\|\eta\|_{\mathcal M_2}^{\frac{\epsilon}{2}}
		\Big)
		\nonumber\\
		&\leq
		C
		\|S^{\alpha,\beta}(t)\xi_{1}
		-S^{\alpha,\beta}(t)\xi_{2}\|_{\mathcal H}^{\frac{\epsilon-\delta_0}{4-\delta_0}} .
	\end{align}
	Consequently, the mapping $
	S^{\alpha,\beta}(1):\mathcal A_{\exp}^{\alpha,\beta}	\subset\mathcal H	\to \mathcal A_{\exp}^{\alpha,\beta}\subset\mathcal Z_{\epsilon}$	is $\frac{\epsilon-\delta_0}{4-\delta_0}$-H\"older continuous.
	
	For the  case \(\beta=0\), we distinguish two cases according to the damping exponent. When $\beta=0$ and $\alpha=2$, the above interpolation argument can be applied without the memory component, which yields $$S^{2,0}(1):
	\mathcal A_{\exp}^{2,0}\subset V_2\times L^2(\Omega)\to\mathcal A_{\exp}^{2,0}\subset V_{4-\epsilon}\times V_{4-\epsilon}$$	being \(\frac{\epsilon-\delta_0}{4-\delta_0}\)-H\"older continuous. For \(\beta=0\) and \(\alpha\in\lbrack0,2)\), following an argument analogous to that
	used in \eqref{3.5}, one can obtain
	\begin{align}\label{4.6-}
		\|S^{\alpha,0}(t)\xi_1 -S^{\alpha,0}(t)\xi_2\|_{V_4\times V_{4-\varepsilon}}
		\leq	C	\|S^{\alpha,0}(t)\xi_1-S^{\alpha,0}(t)\xi_2\|_{\mathcal H}^{\rho}.
	\end{align}
 Therefore, $	S^{\alpha,0}(1):	\mathcal {A}_{\exp}^{\alpha,0}\subset V_2\times L^2(\Omega)\to\mathcal{A}_{\exp}^{\alpha,0}	\subset	V_{4}\times V_{4-\epsilon}$	is $\rho-$ H\"older continuous.
	
	Thus, $\mathcal{E}_{\exp}^{\alpha,\beta}$ is a forward invariant compact set in $\mathcal{Z}_{\epsilon}$ and has finite dimension
 \begin{align*}
 	\mbox{dim}_{f}\left(\mathcal{E}_{exp}^{\alpha,},\mathcal{Z}_{\epsilon}\right)\le k^{-1}\mbox{dim}_{f}\left(\mathcal{A}_{exp}^{\alpha,\beta},\mathcal{H}\right)<\infty,
 \end{align*}
for some constant  $k=\min\{\rho,\frac{\epsilon-\delta_{0}}{4-\delta_{0}}\}$. Moreover, for any above  bounded set $D\subset \mathcal{H}$,  we get from \eqref{4.11}-\eqref{4.6-} that for $t>t_{D}+1$,
 \begin{align*}
 	\mbox{dist}_{\mathcal{Z}_{\epsilon}}\left(S^{\alpha,\beta}(t)D, \mathcal{E}_{exp}^{\alpha,\beta}\right)=& \mbox{dist}_{\mathcal{Z}_{\epsilon}}\left(S^{\alpha,\beta}(1)S^{\alpha,\beta}(t-1)D, S^{\alpha,\beta}(1)\mathcal{A}_{exp}^{\alpha,\beta}\right)\nonumber\\
 	\le& C \left[ \mbox{dist}_{\mathcal{H}}\left(S^{\alpha,\beta}(t-1)D, \mathcal{A}_{exp}^{\alpha,\beta}\right)\right]^{k}\le Ce^{-\widetilde{\kappa} (t-t_{D})},
 \end{align*}
 where $\widetilde{\kappa}>0$ is a constant. Hence, $\mathcal{E}_{exp}^{\alpha,\beta}$ is a strong $(\mathcal{H}, \mathcal{Z}_{\epsilon})- $ exponential attractor of the semigroup $S^{\alpha,\beta}(t)$.

Next, we  prove the continuity of the strong $(\mathcal{H}, \mathcal{Z}_{\epsilon})$-exponential attractor with respect to $(\alpha_{0},\beta_{0}) \in \Xi$. We define the functional
\begin{align*}
	\Gamma((\alpha,\beta),(\alpha_{0},\beta_{0}))=\sup_{\xi\in \mathcal{B}_{R}}\|V_{\alpha,\beta}^{k}\xi-V_{\alpha_{0},\beta_{0}}^{k}\xi\|_{\mathcal{H}}\to 0, ~~\mbox{as}~~(\alpha,\beta)\to(\alpha_{0}, \beta_{0}).
\end{align*}
which follows from \eqref{key0}, where $\mathcal{B}_{R}$ is as defined in \eqref{key1}. Therefore, by Lemma \ref{lem Yang}, the exponential attractor $M_{\alpha,\beta}$ of the dynamical system $(\mathcal{B}_{R}, V_{\alpha,\beta}^{k})$ is continuous with respect to $(\alpha_{0},\beta_{0}) \in\Xi$. More precisely,
\begin{align}\label{4.4}
	\mbox{dist}_{\mathcal{H}}^{symm}(M_{\alpha,\beta}, M_{\alpha_{0},\beta_{0}})\le C(	\Gamma((\alpha,\beta),(\alpha_{0},\beta_{0})))^{\mu}\to 0, \mbox{~as~}(\alpha,\beta)\to(\alpha_{0}, \beta_{0}),
\end{align}
where $0 < \mu < 1$. For any $\xi_{u^{\alpha,\beta}}\in \mathcal{E}_{\exp}^{\alpha,\beta}$, there exist $\xi \in M_{\alpha,\beta}$ and $t\in[1,T+1]$ such that  $\xi_{u^{\alpha,\beta}}=S^{\alpha,\beta}(t)\xi$. Then, by virtue of \eqref{key0}, \eqref{7.2}, \eqref{4.6}, \eqref{4.6-} and \eqref{4.4}, we obtain
\begin{align*}
	&\mbox{dist}_{\mathcal{Z}_{\epsilon}}\left(\mathcal{E}_{exp}^{\alpha,\beta}, \mathcal{E}_{exp}^{\alpha_{0},\beta_{0}}\right)\nonumber\\
	=&\sup_{\xi_{u^{\alpha,\beta}}\in \mathcal{E}_{\exp}^{\alpha,\beta}}\inf_{\xi_{u^{\alpha_{0},\beta_{0}}}\in \mathcal{E}_{\exp}^{\alpha_{0},\beta_{0}}}\|\xi_{u^{\alpha,\beta}}-\xi_{u^{\alpha_{0},\beta_{0}}}\|_{\mathcal{Z}_{\epsilon}}\nonumber\\
	\le&\sup_{t\in[1,T+1]} \sup_{\xi\in M_{\alpha,\beta}}\inf_{\xi^{'}\in M_{\alpha_{0},\beta_{0}}}\|S^{\alpha,\beta}(t)\xi-S^{\alpha_{0},\beta_{0}}(t)\xi^{'}\|_{\mathcal{Z}_{\epsilon}}\nonumber\\
	\le&\sup_{t\in[1,T+1]}\sup_{\xi\in \mathcal{B}_{R}}\|S^{\alpha,\beta}(t)\xi-S^{\alpha_{0},\beta_{0}}(t)\xi\|_{\mathcal{Z}_{\epsilon}}\nonumber\\
	&+\sup_{t\in[1,T+1]}\sup_{\xi\in M_{\alpha,\beta}}\inf_{\xi^{'}\in M_{\alpha_{0},\beta_{0}}}\|S^{\alpha_{0},\beta_{0}}(t)\xi-S^{\alpha_{0},\beta_{0}}(t)\xi^{'}\|_{\mathcal{Z}_{\epsilon}}\nonumber\\
	\le&\sup_{t\in[1,T+1]}\sup_{\xi\in \mathcal{B}_{R}}\|S^{\alpha,\beta}(t)\xi-S^{\alpha_{0},\beta_{0}}(t)\xi\|_{\mathcal{Z}_{\epsilon}}+Ce^{\frac{k}{2}T}\sup_{\xi\in M_{\alpha,\beta}}\inf_{\xi^{'}\in M_{\alpha_{0},\beta_{0}}}\|\xi-\xi^{'}\|_{\mathcal{H}}^{k}\nonumber\\
	\le& \sup_{t\in[1,T+1]}\sup_{\xi\in \mathcal{B}_{R}}\|S^{\alpha,\beta}(t)\xi-S^{\alpha_{0},\beta_{0}}(t)\xi\|_{\mathcal{Z}_{\epsilon}}+Ce^{\frac{k}{2}T}\left[\mbox{dist}_{\mathcal{H}}(M_{\alpha,\beta}, M_{\alpha_{0},\beta_{0}})\right]^{k}\to 0
\end{align*}
as $(\alpha,\beta)\to( \alpha_{0},\beta_{0})$.

 For fixed $\beta\equiv 0$, the continuity of the family of exponential attractors with respect to $\alpha_{0}\in\lbrack0,2)$ follows from \eqref{r} and \eqref{4.6-} by an argument analogous to that used above, yielding \eqref{4.33}. This completes the proof of the theorem.

\end{proof}

\vspace{0.6cm}

{\bf Acknowledgements}

The work was supported partly by the NSF of China (12171094), the Shanghai Key Laboratory for Contemporary Applied Mathematics (08DZ2271900), and Fuyang Normal University (2025KYQD0157,2025AHGXZK40563), Outstanding Youth Research Project of Anhui Universities (2023AH030075).

\vspace{0.4cm}

{\bf Contributions}

The authors contributed equally and significantly in writing this paper.

\vspace{0.4cm}

{\bf Data Availability}

Data sharing is not applicable to this article as no datasets were generated or analyzed during the current study.

\vspace{0.4cm}

{\bf Conflict of interest}

The authors declare that they have no conflict of interest.

\vspace{0.4cm}

{\bf Ethical Statement}

There are no ethical concerns applicable to our research.

\end{document}